\documentclass[11pt,a4paper,reqno]{amsart}
\usepackage[left=3cm,right=3cm,top=2.8cm,bottom=2.8cm]{geometry}
\usepackage[english]{babel}

\usepackage{amsmath,amsfonts,amsthm,amssymb}
\usepackage{cancel}
\usepackage{latexsym,soul,cite,mathrsfs}
\usepackage{cite,marginnote}
\usepackage[hyperpageref]{backref}
\usepackage{todonotes}
\usepackage{cases}

\usepackage{color,enumitem,graphicx}
\usepackage[colorlinks=true,urlcolor=blue,
citecolor=red,linkcolor=blue,linktocpage,pdfpagelabels]{hyperref}
\usepackage{tikz}

\numberwithin{equation}{section}
\newtheorem{theorem}{Theorem}[section]
\newtheorem{lemma}{Lemma}[section]

\newtheorem{proposition}{Proposition}[section]
\newtheorem{remark}{Remark}[section]

\makeindex

\newcommand{\R}{\mathbb R}
\newcommand{\N}{\mathbb N}

\def\e{{\rm e}}
\def\ri{{\rm i}}

 \def\dd{\, {\rm d}}

\title[]{Mixed-dispersion Schr\"odinger equations and Gagliardo-Nirenberg inequalities: equivalence between ground states and optimizers}
\author[Z.Liu, G.Romani, Y.Su]{Zhisu Liu, Giulio Romani, Yu Su}

\address[Zhisu Liu]{\newline\indent
School of Mathematics and Physics, China University of Geosciences,
\newline\indent
Wuhan, Hubei, 430074, P.R. China\\
\newline\indent
Institute for Advanced Marine Research, China University of Geosciences,
\newline\indent
Guangzhou 511462, P. R. China}
\email{\href{mailto:liuzhisu@cug.edu.cn}{liuzhisu@cug.edu.cn}}

\address[Giulio Romani]{\newline\indent
Dipartimento di Scienze Matematiche, Informatiche e Fisiche,
\newline\indent
Universit\`{a} degli Studi di Udine,
\newline\indent
Via delle Scienze 206, 33100 Udine, Italy.}
\email{\href{mailto:giulio.romani@uniud.it}{giulio.romani@uniud.it}}

\address[Yu Su]{\newline\indent
School of Mathematics and Big Data,
\newline\indent
Anhui University of Science and Technology,
\newline\indent
Huainan, Anhui 232001, PR China.}
\email{\href{mailto:yusumath@aust.edu.cn}{yusumath@aust.edu.cn}}


\begin{document}

\begin{abstract}
    We study a nonlinear Schr\"odinger equation with mixed dispersion in the mass competition regime, namely mass-supercritical for the Laplacian and mass-subcritical for the Bilaplacian. In this setting, the existence of a critical value of the mass $c_\varepsilon$, which divides existence and nonexistence of energy ground state solutions, was established in \cite{Bonheure-Casteras-dosSantos-Nascimento2018SIAM}. In this work, we strengthen these results by investigating the relationship between the energy ground states with critical mass, and the optimizers of mixed Gagliardo-Nirenberg-type inequalities. Moreover, we discuss the equivalence between energy and action ground states solutions.
\end{abstract}

\maketitle

\section{Introduction}
In this paper, we study properties of standing waves solutions of the nonlinear Schr\"odinger equation with fourth-order dispersion
\begin{equation}\label{1.1}
	\begin{cases}
		\ri\frac{\partial\psi}{\partial t}-\varepsilon\Delta^2\psi+\Delta\psi+|\psi|^{p-2}\psi=0\,, \quad\ (t,x)\in \R^+\times\R^N,\\
		\psi(0,x)=\psi_0(x)\,,
	\end{cases}
\end{equation}
where we always assume $\varepsilon>0$ and $p\in(2,2^*)$, where $2^*=\frac{2N}{N-4}$ if $N\geqslant5$, while $2^*=+\infty$ otherwise. Inserting the ansatz
\begin{equation}\label{psi}
	\psi(t,x)=\e^{\ri\omega t}u(x)\,,\quad\ \omega\in\R\,,
\end{equation}
in \eqref{1.1}, it is easy to verify that $\psi$ is a solution to \eqref{1.1} if and only if $u$ is a solution to the fourth-order elliptic nonlinear Schr\"{o}dinger equation
\begin{equation}\label{BS}\tag{$BS$}
	\varepsilon\Delta^2 u-\Delta u+\omega u=|u|^{p-2}u\,, \quad\ x\in \R^N.
\end{equation}
Since the mass $\|\psi(t,\,\cdot)\|_{L^2(\R^N)}$ is a conserved quantity of the problem \eqref{1.1}, it is natural to look for solutions $u$ of \eqref{BS} with prescribed $L^2$-norm. From a physical point of view, the most interesting solutions are the {\it energy ground states (energy GSS)} namely those solutions which minimize the energy (that is the second quantity which is conserved in time by \eqref{1.1}), subject to a mass constraint. More precisely, energy GSS are the minimizers of the problem
\begin{equation}\label{m_eps_c}
	m_\varepsilon(c):=\inf_{u\in S_{c}}E_\varepsilon(u)\,,
\end{equation}
where
\begin{equation*}
	S_c:=\left\{u\in H^2(\R^N)\bigg| \int_{\R^N}|u|^2\dd x=c\right\}
\end{equation*}
and the \textit{energy functional} is
\begin{equation*}
	E_\varepsilon(u):=\frac\varepsilon2\int_{\R^N}|\Delta u|^2\dd x+\frac12\int_{\R^N}|\nabla u|^2\dd x-\frac1p\int_{\R^N}|u|^p\dd x\,.
\end{equation*}
In this way, to each minimizer $u\in S_c$ of $m_\varepsilon(c)$ corresponds a frequency $\omega>0$ (which is therefore part of the unknown and found as a Lagrange multiplier) such that $(u,\omega)$ weakly solves \eqref{BS}. 

If one instead prescribes a fixed frequency $\omega>0$ for the standing wave $\psi$ in \eqref{psi}, one may find solutions of \eqref{BS} by looking for critical points of the \textit{action functional}
\begin{equation*}
	\begin{split}
		I_{\omega}(u)&=\frac\varepsilon2\int_{\R^N}|\Delta u|^2\dd x+\frac12\int_{\R^N}|\nabla u|^2\dd x+\frac\omega2\int_{\R^N}|u|^2\dd x-\frac1p\int_{\R^N}|u|^p\dd x\\
		&=E_\varepsilon(u)+\frac\omega2\|u\|_{L^2(\R^N)}^2\,.
	\end{split}
\end{equation*}
Recalling that the Nehari manifold
\begin{equation}\label{Nehari}
	\mathcal{N}:=\{u\in H^2(\R^N)\backslash\{0\}\,|\,\langle I_\omega'(u),u\rangle=0\}
\end{equation}
is a natural constraint for critical points of $I_\omega$, the \textit{action GSS} are defined as the minimizers of the problem
\begin{equation}\label{sigma_omega}
	\sigma_{\omega}:=\inf_{u\in \mathcal{N}}I_\omega(u)\,.
\end{equation}
It is clear that now is the mass of the solution, namely its $L^2$ norm, that cannot be a priori fixed.

\vskip0.2truecm
Main aim of this paper is to sharpen the understanding of energy GSS and action GSS, by investigating the relationships between the two concepts, as well as, by characterizing them with respect to suitable Gagliardo-Nirenberg type inequalities. Before entering into the details of our results, and in order to motivate them, let us recall some important works in this direction.

\vskip0.2truecm

If $\varepsilon=0$, \eqref{1.1} reduces to the well-known nonlinear Schr\"odinger equation
\begin{equation}\label{NLS}\tag{$NLS$}
	\begin{cases}
		i\frac{\partial \psi}{\partial t}+\Delta \psi+|\psi|^{p-2}\psi=0\,, \quad\ (t,x)\in \R^+\times\R^N,\\
		\psi(0,x)=\psi_0(x)\,,
	\end{cases}
\end{equation}
which serves as a model for many important phenomena in physics, such as nonlinear optics and the theory of water waves. The NLS is nowadays well understood from a mathematical point of view. In \cite[Corollary 4.3.4]{Cazenave2003BOOK} Cazenave proved that \eqref{NLS} is locally well-posed in $H^1(\R^N)$ for $p<\frac{2N}{N-2}$ if $N\geqslant3$ or $p<+\infty$ if $N\in\{1,2\}$; moreover, when $p<2+\frac4N$, all solutions to \eqref{NLS} exist globally in time \cite[Corollary 6.1.2]{Cazenave2003BOOK} and standing waves are orbitally stable \cite[Theorem 8.3.1]{Cazenave2003BOOK}. When $p\geqslant2+\frac4N$, then finite-time blowup may occur and the standing wave solutions become unstable, see \cite[Theorem 6.5.10]{Cazenave2003BOOK}. For the applications, this is unpleasant, in particular since the Kerr nonlinearity ($p=4$) turns out to be critical in dimension $N=2$ and supercritical when $N=3$. 

In order to recover stability of the solutions to \eqref{NLS} for $p\geqslant2+\frac4N$, Karpman and Shagalov proposed in \cite{Karpman-Shagalov2000PD} to add a fourth-order dispersion term in the model, i.e., they considered equation \eqref{1.1} with $\varepsilon>0$ small. In fact, they showed that the standing waves are orbitally stable for any $\varepsilon>0$ when $p\leqslant 2+\frac4N$, and for small $\varepsilon>0$ when $2+\frac4N<p<2+\frac8N$, while the instability occurs when $p\geqslant2+\frac8N$. In this way, one retrieves stability for \eqref{NLS} with the cubic Kerr nonlinearity in the physical dimensions $N=2,3$. 

The introduction of a small biharmonic term can be also justified if one comes back to the derivation of the \eqref{NLS} from the nonlinear Helmholtz equation in nonlinear optics, by means of the paraxial approximation: one recovers \eqref{1.1} if one does not neglect the influence of a small but nonzero biharmonic term which arises as (part of) the nonparaxial correction to the \eqref{NLS}. With this additional term, the authors in \cite{Fibich-Ilan-Papanicolaou2002SIAMJAM} managed indeed to show that solutions to \eqref{1.1} exist globally in time for any $p<2+\frac8N$. For a more detailed insight on the topic and related results, we refer to \cite{Bonheure-Casteras-Gou-Jeanjean2019IMRN,Bonheure-Casteras-Mandel2019JLMS,Boulenger-Lenzmann2017AENS,Dinh2021Nonlinearity,Campos-Guzan2022CVPDE,LuoX-YangT2023SCM,Pausader2007DPDE,Pausader-ShaoSL2010JHDE,Pausader2009JFA,Saanouni2021CVPDE,Fibich-Ilan-Schochet2003Nonlinearity}.

\vskip0.2truecm
Let us first consider the \lq\lq energy approach\rq\rq, namely problem \eqref{m_eps_c}. In order to better introduce our analysis of energy GSS for \eqref{BS}, let us start by recalling some results on the second-order Schr\"{o}dinger equation
\begin{equation}\label{S}\tag{$S$}
	-\Delta u+\omega u=|u|^{p-2}u\,, \quad\ x\in \R^N,
\end{equation}
which comes from \eqref{NLS}, and is \eqref{BS} for $\varepsilon=0$. The mass-critical exponent of equation \eqref{S} is
\begin{equation*}
	p=2+\frac4N\,,
\end{equation*}
which can be derived by means of the classical Gagliardo-Nirenberg inequality \cite{N}: for $N\geqslant3$, $p\in(2,\frac{2N}{N-2})$ and $u\in H^1(\R^N)$, one has
\begin{equation}\label{1.2}
	\|u\|_{L^p(\R^N)}^p\leqslant C\|u\|_{D^{1,2}(\R^N)}^{\frac{N(p-2)}2}\|u\|_{L^2(\R^N)}^{p-\frac{N(p-2)}2},
\end{equation}
where the space $D^{1,2}(\R^N)$ is defined as the completion of $C^\infty_0(\R^N)$ with seminorm
\begin{equation*}
	\|u\|_{D^{1,2}(\R^N)}^2:=\int_{\R^N}|\nabla u|^2\dd x\,.
\end{equation*}

By using variational methods and \eqref{1.2}, Cazenave and Lions proved in \cite{Cazenave-Lions1982CMP} the existence of energy GSS for equation \eqref{S} with mass-subcritical exponent $p\in(2,2+\frac4N)$, and their nonexistence in the mass-supercritical regime $p\in(2+\frac4N,\frac{2N}{N-2})$. The mass-critical case $p=2+\frac4N$ was then considered by Weinstein in \cite{Weinstein1983CMP}. Their results may be summarized as follows.

\newtheorem{TheoremA}{Theorem}
\renewcommand{\theTheoremA}{A}
\begin{TheoremA}
	{\rm \cite{Cazenave-Lions1982CMP,Weinstein1983CMP}}
	Let $N\geqslant 1$.
	\begin{enumerate}
		\item If $p\in(2,2+\frac4N)$, there exists an energy GSS for any $c>0$.
		\item If $p=2+\frac4N$, there exists $\bar{c}>0$ such that \eqref{S} has an energy GSS if and only if $c=\bar c$. Moreover, one can construct an energy GSS by means of the optimizer of the Gagliardo-Nirenberg inequality \eqref{1.2}.
		\item If $p\in(2+\frac4N,\frac{2N}{N-2})$, there is no energy GSS for any $c>0$.
	\end{enumerate}
\end{TheoremA}

A similar analysis can be carried out for the biharmonic equation
\begin{equation}\label{B}\tag{$B$}
	\Delta^2 u+\omega u=|u|^{p-2}u\,, \quad\ x\in \R^N,
\end{equation}
which is \eqref{BS} with $\varepsilon=1$ but without the second-order term $-\Delta u$, see \cite[page 9]{Fernandez-Jeanjean-Mandel-Maris2022JDE}. Now the mass-critical exponent of \eqref{B} is
\begin{equation*}
	p=2+\frac8N\,,
\end{equation*}
which can be derived by means of the second-order analogue of the Gagliardo-Nirenberg inequality \eqref{1.2}: for $q\in(2,\frac{2N}{N-4})$ and $u\in H^2(\R^N)$, then
\begin{equation}\label{1.3}
	\|u\|_{L^q(\R^N)}^q\leqslant C\|u\|_{D^{2,2}(\R^N)}^{\frac{N(q-2)}4}\|u\|_{L^2(\R^N)}^{q-\frac{N(q-2)}4}\,,
\end{equation}
where the space $D^{2,2}(\R^N)$ is defined as the completion of $C^\infty_0(\R^N)$ with seminorm
\begin{equation*}
	\|u\|_{D^{2,2}(\R^N)}^2:=\int_{\R^N}|\Delta u|^2\dd x\,.
\end{equation*}
Also in this setting, in the mass-critical case $p=2+\frac8N$, the phenomenon of a critical mass appears, as the only value of the mass for which energy GSS are allowed.

Compared to \eqref{NLS} and \eqref{B}, the \textit{mixed dipersion} nonlinear Schr\"odinger equation \eqref{BS} presents intrinsic
difficulties due to the interplay between the Laplacian and the Bilaplacian, especially in the \textit{mass competition regime}
\begin{equation*}
	p\in\left(2+\frac4N,2+\frac8N\right),
\end{equation*}
that is, when $p$ is mass-subcritical for problem \eqref{B} and mass-supercritical for problem \eqref{S}. In this setting, Bonheure et al. in \cite{Bonheure-Casteras-Gou-Jeanjean2019TAMS,Bonheure-Casteras-dosSantos-Nascimento2018SIAM} proved the following mixed Gagliardo-Nirenberg-type inequality:
\begin{equation}\label{1.4}\tag{GN$_{C}$}
	\|u\|_{L^p(\R^N)}^p\leqslant C_{N,p}\|u\|_{D^{2,2}(\R^N)}^{\frac{N(p-2)-4}2}\|u\|_{D^{1,2}(\R^N)}^{\frac{8-N(p-2)}2}\|u\|_{L^2(\R^N)}^{p-2}\,,\qquad u\in H^2(\R^N)\,.
\end{equation}
Here $C_{N,p}$ denotes the best constant of \eqref{1.4}. In light of this, they investigated existence and nonexistence of energy GSS for equation \eqref{BS}, and in the next theorem we state their main results.

\newtheorem{TheoremB}{Theorem}
\renewcommand{\theTheoremB}{B}
\begin{TheoremB}\label{TheoremB}
	{\rm\cite{Bonheure-Casteras-Gou-Jeanjean2019TAMS,Bonheure-Casteras-dosSantos-Nascimento2018SIAM}}
	Let $\varepsilon>0$. Then there exists $c_\varepsilon>0$ such that the following holds.
	\begin{enumerate}
		\item If $p\in(2,2+\frac4N)$, then equation \eqref{BS} has an energy GSS for any $c>0$.
		\item If $p=2+\frac4N$, equation \eqref{BS} has an energy GSS for $c>c_\varepsilon$, and has no energy GSS for $c<c_\varepsilon$.
		\item If $p\in(2+\frac4N,2+\frac8N)$, then equation \eqref{BS} has an energy GSS if and only if $c\geqslant c_\varepsilon$, where
		\begin{equation}\label{c_epsilon}
			c_\varepsilon=C_{N,p}^{-\frac2{p-2}}\left[\frac{2p}{N(p-2)-4}\right]^{\frac{N(p-2)-4}{2(p-2)}}\left[\frac{2p}{8-N(p-2)}\right]^{\frac{8-N(p-2)}{2(p-2)}}\varepsilon^{\frac{N(p-2)-4}{2(p-2)}}.
		\end{equation}
		\item If $p\in[2+\frac8N,\frac{2N}{N-4})$, then there is no energy GSS for any $c>0$.
	\end{enumerate}
\end{TheoremB}
For more recent advances in the topic, we refer to \cite{LTW,LZZ,LuoX-YangT2023SCM,Ma}.
\begin{remark}
	In these results, the fourth-order dispersion parameter $\varepsilon$ is given, and the condition for the existence of energy GSS is described by the size of the mass $c$. From the physical point of view, Karpman \cite{Karpman1996PRE} and Karpman-Shagalov \cite{Karpman-Shagalov2000PD} considered the stability of steady solitions by the effect of the fourth-order dispersion $\varepsilon$, hence, it is rather natural to describe the condition for the existence of energy GSS in terms of $\varepsilon$. The two approaches are clearly equivalent, and we can therefore restate the third statement of Theorem \ref{TheoremB} as follows:
	Let $N\geqslant5$, $p\in(2+\frac4N,2+\frac8N)$ and $c>0$ be given. Then equation \eqref{BS} has an energy GSS if and only if $\varepsilon\in(0,\varepsilon_c]$, where
	\begin{equation*}
		\varepsilon_c:=c^{\frac{p-2}{N(p-2)-4}}C_{N,p}^{\frac4{N(p-2)-4}}\frac{N(p-2)-4}{2p}\left[\frac{2p}{8-N(p-2)}\right]^{-\frac{8-N(p-2)}{N(p-2)-4}}.
	\end{equation*}
\end{remark}

\vskip0.2truecm
Let us now come to the complementary approach of finding solutions by minimizing action functionals on the corresponding Nehari set. For elliptic problems of the kind \eqref{NLS} or \eqref{BS}, this method has a long history and has nowadays become standard, and it is well known that action GSS exist for \eqref{BS} for any $p\in(2,2^*)$. We just mention, among others, the recent works \cite{BN,d'Avenia-Pomponio-Schino2023Nonlinearity} concerning \eqref{BS}. It is interesting therefore to compare the results of the \lq\lq energy\rq\rq approach and the \lq\lq action\rq\rq approach. In other words, under which conditions the two sets of solutions coincide? For \eqref{NLS} this long-standing problem was firstly solved in \cite{JL} in the $\R^N$ setting, where the equivalence between energy GSS and action GSS was shown for $p\in(2,2+\frac4N)$. We point out that the question for bounded domains is more delicate, and we refer to \cite{Dovetta-Serra-Tilli2023MA}; we also mention that the uniqueness of energy GSS for \eqref{NLS} when $p\in (2,2+\frac4N)$ is established in \cite{Hajaiej-SongLJ2025MZ,Dovetta-Serra-Tilli2020AM}. Concerning \eqref{BS} but with a positive sign in front of the Laplacian, equivalence results between energy GSS and action GSS have been established in \cite[Theorem 1.3]{Fernandez-Jeanjean-Mandel-Maris2022JDE} for $p\in(2,2+\frac8N)$.

\subsection{Main Results}

In Theorem A, for $p=2+\frac4N$ and critical mass $\bar c$, the authors characterized  $u_{\bar c}$ (the energy GSS with critical mass $\bar c$), i.e., they showed the relation between the optimizer of the Gagliardo-Nirenberg inequality \eqref{1.2} and $u_{\bar c}$. A corresponding result can also be proved for equation \eqref{B} in the mass-critical regime $p=2+\frac8N$. For $p\in(2+\frac4N,2+\frac8N)$, such a connection between energy GSS of \eqref{BS} with critical mass and the optimizers of the corresponding inequality \eqref{1.4} has not been investigated in \cite{Bonheure-Casteras-dosSantos-Nascimento2018SIAM}. Our first aim is therefore to characterize the energy GSS $u_{c_\varepsilon}$ in terms of optimizers of suitable Gagliardo-Nirenberg inequalities.

We start by recalling the following non-homogeneous Gagliardo-Nirenberg inequality (see \cite{Fernandez-Jeanjean-Mandel-Maris2022JDE})
\begin{equation}\label{1.5}\tag{GN$_{ K}$}
	\|u\|_{L^p(\R^N)}^p\leqslant K_{N,p,\varepsilon}\|u\|_{L^2(\R^N)}^{p-2}\left(\varepsilon\|u\|_{D^{2,2}(\R^N)}^{2}+\|u\|_{D^{1,2}(\R^N)}^{2}\right),
\end{equation}
where $\varepsilon>0$, $p\in (2,\frac{2N}{N-4})$, $u\in H^2(\R^N)$, and $K_{N,p,\varepsilon}>0$ is the best constant in \eqref{1.5}. First, inspired by the method of Fernandez-Jeanjean-Mandel-Maris \cite[Page 5 and (1.5)]{Fernandez-Jeanjean-Mandel-Maris2022JDE}, in our first result, we characterize the energy GSS and its mass in terms of the non-homogeneous inequality \eqref{1.5}.
\begin{theorem}\label{Theorem1.1}
	Let $p\in(2+\frac4N,2+\frac8N)$ and $\varepsilon>0$. Then $u$ is the energy GSS with critical mass $c_\varepsilon$ if and only if  $\|u\|_{L^2(\R^N)}^2=c_\varepsilon$ and $u$ is an optimizer of  \eqref{1.5}. Moreover,
	\begin{equation*}
		c_\varepsilon=\left(\frac2pK_{N,p,\varepsilon}\right)^{-\frac2{p-2}}.
	\end{equation*}
\end{theorem}

\begin{remark}
	Note that, as a consequence of Theorems \ref{TheoremB} and \ref{Theorem1.1}, we have the following characterization of the best constant $K_{N,p}$ of inequality \eqref{1.5}:
	\begin{equation*}
		K_{N,p,\varepsilon}=\frac p2C_{N,p}\left[\frac{2p}{N(p-2)-4}\right]^{-\frac{N(p-2)-4}4}\left[\frac{2p}{8-N(p-2)}\right]^{-\frac{8-N(p-2)}{2(p-2)}}\varepsilon^{-\frac{N(p-2)-4}4}.
	\end{equation*}
\end{remark}

Then, we characterize $u_{c_\varepsilon}$ by means of the homogeneous Gagliardo-Nirenberg inequality \eqref{1.4} as follows.
\begin{theorem}\label{Theorem1.2}
	Let $p\in(2+\frac4N,2+\frac8N)$ and $\varepsilon>0$. Then $u$ is an energy GSS with critical mass $c_\varepsilon$ if and only if $u$ is an optimizer of inequality \eqref{1.4} with $\|u\|_{L^2(\R^N)}^2=c_\varepsilon$ defined in \eqref{c_epsilon}.
\end{theorem}

We underline that, compared to the second-order case, e.g. \cite{Weinstein1983CMP}, we cannot relate anymore on the radial symmetry of the optimizers of inequality \eqref{1.4}. Indeed, the usual methods to prove symmetry, such as the moving-planes method or rearrangement techniques are not at disposal, since both are ultimately based on the maximum principle, which in general fails for higher-order problems. We refer the interested reader to \cite{GGS} for a comprehensive discussion on this topic, and to the more recent advances in \cite{CT,GRS}. Recently, Lenzmann and Sok proposed in \cite{Lenzmann-Sok2021IMRN}, see also \cite{Lenzmann-Weth2023JAM}, a new strategy by considering rearrangements in Fourier space, which is efficient, provided $p>2$ is an even integer or $p=+\infty$, so in general we cannot rely on this technique. As a result, in $H^2(\R^N)$ without symmetry information, it not elementary to avoid vanishing and dichotomy when dealing with the concentration-compactness principle, see e.g. \cite{FY}.

\vskip0.2truecm
Finally, in the spirit of \cite[Theorem 1.3]{Fernandez-Jeanjean-Mandel-Maris2022JDE}, we analyze the relationship between energy GSS and action GSS of \eqref{BS}: in the mass competition case we prove that for any $\varepsilon>0$ the two sets of solutions are the same.

\begin{theorem}\label{Theorem1.3}
	Let $p\in(2+\frac4N,2+\frac8N)$, $\varepsilon>0$, and
	\begin{equation}\label{omega}
		\omega(\varepsilon):=(p-2)\frac{N(p-2)-4}2\left(\frac{8-N(p-2)}2\right)^{-2}\|v\|_{L^2(\R^N)}^{-2}\varepsilon^{-1},
	\end{equation}
	where $v$ is an optimizer of inequality \eqref{1.4} such that $\|\nabla v\|_{L^2(\R^N)}=\|\Delta v\|_{L^2(\R^N)}=1$, as in Lemma \ref{Lemma_minimizer}. Then $u$ is an energy GSS for critical mass $c_\varepsilon$ whose Lagrange multiplier is $\omega=\omega(\varepsilon)$  if and only if $u$ is an action GSS with $\omega=\omega(\varepsilon)$.
\end{theorem}

The results of Theorems \ref{Theorem1.1}-\ref{Theorem1.3} may be summarized with the help of the following diagram.

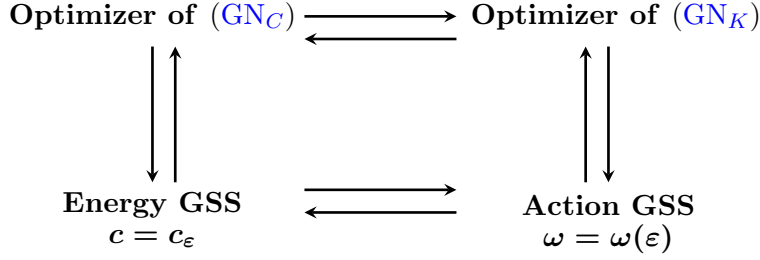
\begin{center}
    \begin{tikzpicture}[>=stealth, line width=1pt]
    
    \node at (-3,2.5) {\textbf{Optimizer of \eqref{1.4}}};
    \node at ( 3.1,2.5) {\textbf{Optimizer of \eqref{1.5}}};
    
    \node at (-3,0) {\textbf{Energy GSS}};
    \node at (-3,-0.45) {$\boldsymbol{c = c_\varepsilon}$};
    
    \node at ( 3,0) {\textbf{Action GSS}};
    \node at ( 3,-0.45) {$\boldsymbol{\omega = \omega(\varepsilon)}$};
    
    \draw[->] (-1,2.5) -- (1,2.5);
    \draw[->] (1,2.2) -- (-1,2.2);
    
    \draw[->] (-1,0.2) -- (1,0.2);
    \draw[->] (1,-0.1) -- (-1,-0.1);
    
    \draw[->] (-3,2.1) -- (-3,0.3);
    \draw[->] (-2.7,0.3) -- (-2.7,2.1);
    
    \draw[->] (3,2.1) -- (3,0.3);
    \draw[->] (2.7,0.3) -- (2.7,2.1);
    
    \end{tikzpicture}
\end{center}

\paragraph{\textbf{Notation}} The norm of the Lebesgue space $L^p(\R^N)$, $p\geqslant1$, is often denoted by $\|\cdot\|_p$, while $\|\cdot\|_{D^{k,2}}:=\|\cdot\|_{D^{k,2}(\R^\N)}$ for $k\in\{1,2\}$.

\section{Preliminaries}

We start by proving the inequality \eqref{1.4} which is fundamental in our work, by combining the Gagliardo-Nirenberg inequalities \eqref{1.2} and \eqref{1.3}. Although this result has been proved \cite[Corollary 2.2]{Bonheure-Casteras-dosSantos-Nascimento2018SIAM}, we give it also here with our notation for the sake of a clearer exposition.
\begin{lemma}\label{Lemma2.1}
	For any $p\in(2+\frac4N,2+\frac8N)$, there exists a constant $C_{N,p}>0$ such that
	\begin{equation*}
		\|u\|_{L^p(\R^N)}^p\leqslant C_{N,p}\|u\|_{D^{2,2}(\R^N)}^{\frac{N(p-2)-4}2}\|u\|_{D^{1,2}(\R^N)}^{\frac{8-N(p-2)}2}\|u\|_{L^2(\R^N)}^{p-2}\,,
	\end{equation*}
	for all $u\in H^2(\R^N)$.
\end{lemma}

\begin{proof}
	By the H\"{o}lder inequality for $p\in(2+\frac4N,2+\frac8N)$, we have
	\begin{equation*}
		\int_{\R^N}|u|^p\dd x\leqslant\left(\int_{\R^N}|u|^{2+\frac4N}\dd x\right)^{\frac{(2+\frac8N)-p}{\frac4N}}\left(\int_{\R^N}|u|^{2+\frac8N}\dd x\right)^{\frac{p-(2+\frac4N)}{\frac4N}},
	\end{equation*}
	and therefore using \eqref{1.2} and \eqref{1.3} with exponents $p=2+\frac4N$ and  $q=2+\frac8N$, respectively, we infer
	\begin{equation*}
		\begin{split}
			\int_{\R^N}|u|^p\dd x&\leqslant C\left(\|u\|_{D^{1,2}(\R^N)}^2\|u\|_{L^2(\R^N)}^{\frac4N}\right)^{\frac{(2+\frac8N)-p}{\frac4N}}\left(\|u\|_{D^{2,2}(\R^N)}^2\|u\|_{L^2(\R^N)}^{\frac8N}\right)^{\frac{p-(2+\frac4N)}{\frac4N}}\\			&=C\|u\|_{D^{2,2}(\R^N)}^{\frac{N(p-2)-4}2}\|u\|_{D^{1,2}(\R^N)}^{\frac{8-N(p-2)}2}\|u\|_{L^2(\R^N)}^{p-2}\,.
		\end{split}
	\end{equation*}
\end{proof}

In the next lemma we collect some results from \cite[Lemmas 2.4 and 2.6]{Bonheure-Casteras-dosSantos-Nascimento2018SIAM} which will be used in our analysis.
\begin{lemma}\label{Lemma_Bonheure}
	Let $p\in(2+\frac4N,2+\frac8N)$ and $\varepsilon>0$. For all $c>0$ one has $m_\varepsilon(c)\leq0$, where $m_\varepsilon(c)$ is defined in \eqref{m_eps_c}. Moreover, $m_\varepsilon(c)=0$ if and only if $c\leqslant c_\varepsilon$, where $c_\varepsilon$ is defined in \eqref{c_epsilon}.
\end{lemma}

\section{First characterization of the energy GSS with critical mass \texorpdfstring{$c_\varepsilon$}{c_epsilon}:\\ Proof of Theorem \ref{Theorem1.1}}

In this section, inspired by the method of Fernandez-Jeanjean-Mandel-Maris \cite[Page 5 and (1.5)]{Fernandez-Jeanjean-Mandel-Maris2022JDE}, we prove Theorem \ref{Theorem1.1}, i.e., the characterization of the energy GSS with critical mass $c_\varepsilon$ via the non-homogeneous inequality \eqref{1.5}.
\vskip0.2truecm

We first note that, for $p\in(2,\frac{2N}{N-4})$ and $\varepsilon>0$, for any $u\in H^2(\R^N)$ with $\|u\|_{L^2(\R^N)}^2=c$ one can rewrite the energy $E_\varepsilon$ as
\begin{equation}\label{3.1}
	E_\varepsilon(u)=\frac12\left(\varepsilon\|u\|_{D^{2,2}}^2+\|u\|_{D^{1,2}}^2\right)\!\left(1-\frac2pc^{\frac{p-2}2}\frac{\,\|u\|_p^p}{\|u\|_2^{p-2}\left(\varepsilon\|u\|_{D^{2,2}}^2+\|u\|_{D^{1,2}}^2\right)}\right).
\end{equation}
Recall from \eqref{1.5} that
\begin{equation*}
	K_{N,p,\varepsilon}=\sup_{u\in H^2\backslash\{0\}}\frac{\|u\|_p^p}{\|u\|_2^{p-2}\left(\varepsilon\|u\|_{D^{2,2}}^{2}+\|u\|_{D^{1,2}}^{2}\right)}\,,
\end{equation*}
for which $0<K_{N,p,\varepsilon}<+\infty$ clearly holds.

\begin{lemma}\label{Lemma3.2}
	Let $p\in(2+\frac4N,2+\frac8N)$ and $\varepsilon>0$. Then
	\begin{equation*}
		K_{N,p,\varepsilon}=\frac{p}2c_\varepsilon^{-\frac{p-2}2}.
	\end{equation*}
	Moreover, the energy GSS $u_{c_\varepsilon}\in H^2\backslash\{0\}$ is an optimizer of inequality \eqref{1.5}.
\end{lemma}

\begin{proof}
	We aim at proving
	\begin{equation*}
		\sup_{u\in H^2\backslash\{0\}}\frac{\|u\|_p^p}{\|u\|_2^{p-2}\left(\varepsilon\|u\|_{D^{2,2}}^{2}+\|u\|_{D^{1,2}}^{2}\right)}=\frac p2c_\varepsilon^{-\frac{p-2}2}.
	\end{equation*}
	We distinguish two cases:
	\begin{enumerate}
		\item [(1)] $\|u\|_2^2=:c=c_\varepsilon$;
		\item [(2)] $\|u\|_2^2=:c\not=c_\varepsilon$.
	\end{enumerate}

	(1) For $c=c_\varepsilon$, by Lemma \ref{Lemma_Bonheure} we know that
	\begin{equation*}
		\inf_{u\in S_{c_\varepsilon}}E_\varepsilon(u)=m_\varepsilon(c_\varepsilon)=0\,.
	\end{equation*}
	Then from \eqref{3.1}, we get that
	\begin{equation*}
		\begin{split}
			0&=\inf_{u\in S_{c_\varepsilon}}\left[1-\frac2pc_\varepsilon^{\frac{p-2}2}\frac{\|u\|_p^p}{\|u\|_2^{p-2}\left(\varepsilon\|\Delta u\|_2^2+\|\nabla u\|_2^2\right)}\right]\\
			&=1-\frac2pc_\varepsilon^{\frac{p-2}2}\sup_{u\in S_{c_\varepsilon}}\frac{\|u\|_p^p}{\|u\|_2^{p-2}\left(\varepsilon\|\Delta u\|_2^2+\|\nabla u\|_2^2\right)}\,,
		\end{split}
	\end{equation*}
	which shows that
	\begin{equation}\label{3.3}
		\sup_{u\in S_{c_\varepsilon}}\frac{\|u\|_p^p}{\|u\|_2^{p-2}\left(\varepsilon\|\Delta u\|_2^2+\|\nabla u\|_2^2\right)}=\frac p2c_\varepsilon^{-\frac{p-2}2}.
	\end{equation}
	From Theorem \ref{TheoremB} and Lemma \ref{Lemma_Bonheure}, we know that there exists $u_{c_\varepsilon}\in H^2(\R^N)\backslash\{0\}$ with $\|u_{c_\varepsilon}\|_2^2=c_\varepsilon$ such that
	\begin{equation*}
		E_\varepsilon(u_{c_\varepsilon})=m_\varepsilon(c_\varepsilon)=0\,.
	\end{equation*}
	It follows from \eqref{3.1} and \eqref{3.3} that
	\begin{equation*}
		\frac{\|u_{c_\varepsilon}\|_p^p}{\|u_{c_\varepsilon}\|_2^{p-2}\left(\varepsilon\|\Delta u_{c_\varepsilon}\|_2^2+\|\nabla u_{c_\varepsilon}\|_2^2\right)}=\frac p2c_\varepsilon^{-\frac{p-2}2}=\sup_{u\in S_{c_\varepsilon}}\frac{\|u\|_p^p}{\|u\|_2^{p-2}\left(\varepsilon\|\Delta u\|_2^2+\|\nabla u\|_2^2\right)}\,.
	\end{equation*}
	
	(2) Now we consider the case $c\not=c_\varepsilon$. For any $v\in H^2(\R^N)\backslash\{0\}$ with $\|v\|_2^2=c\not=c_\varepsilon$, we define
	\begin{equation*}
		u:=\sqrt\frac{c_\varepsilon}c\,v\,,
	\end{equation*}
	so that $u\in S_{c_\varepsilon}$ and, moreover,
	\begin{equation*}
		\frac{\|v\|_p^p}{\|v\|_2^{p-2}\left(\varepsilon\|\Delta v\|_2^2+\|\nabla v\|_2^2\right)}=\frac{\|u\|_p^p}{\|u\|_2^{p-2}\left(\varepsilon\|\Delta u\|_2^2+\|\nabla u\|_2^2\right)}\,.
	\end{equation*}
	Hence, it is clear that
	\begin{equation*}
		\sup_{v\in S_c}\frac{\|v\|_p^p}{\|v\|_2^{p-2}\left(\varepsilon\|\Delta v\|_2^2+\|\nabla v\|_2^2\right)}=\sup_{u\in S_{c_\varepsilon}}\frac{\|u\|_p^p}{\|u\|_2^{p-2}\left(\varepsilon\|\Delta u\|_2^2+\|\nabla u\|_2^2\right)}=\frac p2c_\varepsilon^{-\frac{p-2}2}.
	\end{equation*}
\end{proof}

\begin{proof}[Proof of Theorem \ref{Theorem1.1}]
	From Lemma \ref{Lemma3.2}, we know that the energy GSS with critical mass $c_\varepsilon$ is an optimizer of \eqref{1.5}. On the other hand, if $\|u\|_2^2=c_\varepsilon$ and $u$ is an optimizer of \eqref{1.5}, then from \eqref{3.1} we know that $u$ is an energy GSS.
\end{proof}

\section{Second characterization of the energy GSS with critical mass \texorpdfstring{$c_\varepsilon$}{c_epsilon}:\\ Proof of Theorem \ref{Theorem1.2}}

In this section, we prove Theorem \ref{Theorem1.2}, i.e., the characterization of the energy GSS with critical mass $c_\varepsilon$ via optimizers of \eqref{1.4}. The proof will be complete once establish the following three results, which are proved in the remainder of this section.

\begin{proposition}\label{Theorem4.1}
	Let $p\in(2+\frac4N,2+\frac8N)$ and $\varepsilon>0$. Then there exists an optimizer $Q\in H^2(\R^N)$ for \eqref{1.4}. Moreover, $Q$ is a weak solution of
	\begin{equation}\label{eq_Q'}
		\varepsilon\Delta^2Q-\Delta Q+\omega Q=|Q|^{p-2}Q\,, \quad\ x\in \R^N,
	\end{equation}
	for some $\omega>0$.
\end{proposition}

By virtue of Theorem \ref{Theorem4.1}, we will manage to explicitly characterize the relation between the best constant $C_{N,p}$ in \eqref{1.4} and the critical mass $c_\varepsilon$ in Theorem \ref{TheoremB}. Furthermore, we prove that $Q$ is an energy GSS with critical mass $c_\varepsilon$.

\begin{proposition}\label{Theorem4.2}
	Let $p\in(2+\frac4N,2+\frac8N)$ and $\varepsilon>0$. Then $Q$ given by Proposition \ref{Theorem4.1} is an energy GSS of \eqref{BS} with critical mass, i.e.,
	\begin{equation*}
		E_\varepsilon(Q)=0\qquad\mbox{and}\qquad\|Q\|_2^2=c_\varepsilon\,.
	\end{equation*}
\end{proposition}

\begin{proposition}\label{Theorem4.4}
	Let $p\in(2+\frac4N,2+\frac8N)$ and $\varepsilon>0$.
	Let $u_{c_\varepsilon}\in H^2(\R^N)$ be an energy GSS for \eqref{BS} with mass $\|u_{c_\varepsilon}\|_2^2=c_\varepsilon$. Then
	\begin{equation*}
		\|u_{c_\varepsilon}\|_p^p=C_{N,p}\|u_{c_\varepsilon}\|_{D^{2,2}}^{\frac{N(p-2)-4}2}\|u_{c_\varepsilon}\|_{D^{1,2}}^{\frac{8-N(p-2)}2}\|u_{c_\varepsilon}\|_2^{p-2}\,,
	\end{equation*}
	that is, $u_{c_\varepsilon}$ is an optimizer for \eqref{1.4}.
\end{proposition}

\noindent Henceforth, we always assume $p\in(2+\frac4N,2+\frac8N)$ and $\varepsilon>0$.

\subsection{Existence of optimizers for \texorpdfstring{\eqref{1.4}}{GN_C} (Proof of Proposition \ref{Theorem4.1})}

Inspired by \cite{Weinstein1983CMP}, we define the Weinstein functional associated to problem \eqref{BS} as
\begin{equation*}
	W_p(u):=\frac{\|u\|_{D^{2,2}(\R^N)}^{\frac{N(p-2)-4}2}\|u\|_{D^{1,2}(\R^N)}^{\frac{8-N(p-2)}2}\|u\|_{L^2(\R^N)}^{p-2}}{\|u\|_{L^p(\R^N)}^p}\,,
\end{equation*}
so that
\begin{equation*}
	C_{N,p}^{-1}=\inf_{u\in H^2(\R^N)\backslash\{0\}}W_p(u)\,.
\end{equation*}

\begin{lemma}\label{Lemma_scaling}
	There exists a minimizing sequence $\{v_n\}\subset H^2(\R^N)$ for the Weinstein functional, that is
	\begin{equation}\label{bv_n_Weinstein}
		\lim_{n\to+\infty}W_p(v_n)=C_{N,p}^{-1}\,,
	\end{equation}
	such that
	\begin{equation}\label{bv_n_properties_1}
		\|v_n\|_{D^{2,2}}=1=\|v_n\|_{D^{1,2}}\,,
	\end{equation}
	and moreover
	\begin{equation}\label{bv_n_properties_p}
		\lim_{n\to+\infty}\|v_n\|_p>0
	\end{equation}
	and
	\begin{equation}\label{bv_n_properties_2}
		\lim_{n\to+\infty}\|v_n\|_2\in(0,+\infty)\,.
	\end{equation}
\end{lemma}

\begin{proof}
	From Lemma \ref{Lemma2.1}, we know that there exists a minimizing sequence $\{w_n\}\subset H^2(\R^N)$ for $W_p$, namely such that
	\begin{equation*}
		\lim_{n\to+\infty}W_p(w_n)=C_{N,p}^{-1}\,.
	\end{equation*}
	We rescale the sequence by setting
	\begin{equation*}
		v_n(x)=\Lambda_{1,n}w_n(\Lambda_{2,n}x)\,,
	\end{equation*}
	where
	\begin{equation*}
		\Lambda_{1,n}=\frac{\left(\int_{\R^N}|\nabla w_n|^2\dd x\right)^{\frac{N-4}4}}{\left(\int_{\R^N}|\Delta w_n|^2\dd x\right)^{\frac{N-2}4}}\quad\ \ \mbox{and}\quad\ \  \Lambda_{2,n}=\left(\frac{\int_{\R^N}|\nabla w_n|^2\dd x}{\int_{\R^N}|\Delta w_n|^2\dd x}\right)^{\frac12}.
	\end{equation*}
	Computing the norms involved in the Weinstein functional, we get
	\begin{equation*}
		\int_{\R^N}|\Delta v_n|^2\dd x=\Lambda_{1,n}^2\Lambda_{2,n}^{4-N}\int_{\R^N}|\Delta w_n|^2\dd x=1\,,
	\end{equation*}
	and
	\begin{equation*}
		\int_{\R^N}|\nabla v_n|^2\dd x=\Lambda_{1,n}^2\Lambda_{2,n}^{2-N}\int_{\R^N}|\nabla w_n|^2\dd x=1\,;
	\end{equation*}
	while
	\begin{equation*}
		\int_{\R^N}|v_n|^2\dd x=\Lambda_{1,n}^2\Lambda_{2,n}^{-N}\int_{\R^N}|w_n|^2\dd x\,,
	\end{equation*}
	and
	\begin{equation*}
		\int_{\R^N}|v_n|^p\dd x=\Lambda_{1,n}^p\Lambda_{2,n}^{-N}\int_{\R^N}|w_n|^p\dd x\,.
	\end{equation*}
	By these computations, we easily see that the Weinstein functional $W_p$ is invariant under this scaling. In fact,
	\begin{equation}\label{Wp_invariant}
		\begin{split}
			W_p(v_n)&=\frac{\|v_n\|_{D^{2,2}}^{\frac{N(p-2)-4}2}\|v_n\|_{D^{1,2}}^{\frac{8-N(p-2)}2}\|v_n\|_2^{p-2}}{\|v_n\|_p^p}\\
			&=\frac{[\Lambda_{1,n}^2\Lambda_{2,n}^{4-N}]^{\frac{N(p-2)-4}4}[\Lambda_{1,n}^2\Lambda_{2,n}^{2-N}]^{\frac{8-N(p-2)}4}[\Lambda_{1,n}^2\Lambda_{2,n}^{-N}]^{\frac{p-2}2}}{\Lambda_{1,n}^p\Lambda_{2,n}^{-N}}W_p(w_n)\\
			&=W_p(w_n)\,.
		\end{split}
	\end{equation}
	Consequently, $\{v_n\}$ is also a minimizing sequence. Hence, by \eqref{bv_n_Weinstein} and \eqref{bv_n_properties_1} we infer
	\begin{equation}\label{bv_n_proof}
		C_{N,p}^{-1}=\lim_{n\to+\infty}W_p(v_n)=\lim_{n\to+\infty}\frac{\|v_n\|_2^{p-2}}{\|v_n\|_p^p}\,.
	\end{equation}
	Using \eqref{bv_n_properties_1}, \eqref{bv_n_proof}, and the Gagliardo-Nirenberg inequality \eqref{1.2}, we infer
	\begin{equation*}
		\begin{split}
			C_{N,p}^{-1}&\geqslant\lim_{n\to+\infty}\frac{\|v_n\|_2^{p-2}}{C\|v_n\|_2^{p-\frac{N(p-2)}2}}
			=\frac1C\lim_{n\to+\infty}\|v_n\|_2^{\frac{N(p-2)}2-2}
		\end{split}
	\end{equation*}
	so, since $p>2+\frac4N$, we get the upper bound in \eqref{bv_n_properties_2}. On the other hand, by \eqref{bv_n_properties_1}, \eqref{bv_n_proof}, and the Gagliardo-Nirenberg inequality \eqref{1.3}, one gets
	\begin{equation*}
		\begin{split}
			C_{N,p}^{-1}&\geqslant\lim_{n\to+\infty}\frac{\|v_n\|_2^{p-2}}{C\|v_n\|_2^{p-\frac{N(p-2)}4}}
			=\frac1C\lim_{n\to+\infty}\|v_n\|_2^{\frac{N(p-2)}4-2},
		\end{split}
	\end{equation*}
	which implies the lower bound in \eqref{bv_n_properties_2}, since $p<2+\frac8N$. Finally, we prove the non-vanishing of $\|v_n\|_p$ as follows: by \eqref{bv_n_properties_1} we can write the ratio in \eqref{bv_n_proof} as
	\begin{equation*}
		\begin{split}
			C_{N,p}^{-1}&=\lim_{n\to+\infty}\frac{\left[\|v_n\|_{D^{2,2}}^{\frac{N(p-2)}4}\|v_n\|_2^{p-\frac{N(p-2)}4}\right]^{\frac{4(p-2)}{4p-Np+2N}}}{\|v_n\|_p^p}
			\geqslant\frac1C\lim_{n\to+\infty}\|v_n\|_p^{p\left[\frac{4(p-2)}{4p-Np+2N}-1\right]}\,,
		\end{split}
	\end{equation*}
	using again the Gagliardo-Nirenberg inequality \eqref{1.3}. Since $p<2+\frac8N$, the exponent is negative, hence \eqref{bv_n_properties_p} holds.
\end{proof}

\begin{lemma}\label{Lemma_minimizer}
	There exists a minimizer $v\in H^2(\R^N)\setminus\{0\}$ for inequality \eqref{1.4} such that
	\begin{equation}\label{properties_v}
		\|v\|_{D^{2,2}}^2=\|v\|_{D^{1,2}}^2=1\,,\quad W_p(v)=C_{N,p}^{-1}\,.
	\end{equation}
\end{lemma}

\begin{proof}
	Let $\{\tilde v_n\}\subset H^2(\R^N)$ be the bounded minimizing sequence of $W_p$ given by Lemma \ref{Lemma_scaling}. 
	Note that our problem is translation invariant and $H^2$-subcritical. Therefore, by Lions' concentration–compactness principle \cite{Lions}, there exists $\{y_n\}\subset\R^N$ such that
	\begin{equation}\label{v}
	   v_n:=\tilde v_n(\cdot+y_n)\rightharpoonup v\not\equiv0
	\end{equation}
	weakly in $H^2(\R^N)$. Note indeed that vanishing is excluded by \eqref{bv_n_properties_p}. Clearly $\{v_n\}$ satisfies properties \eqref{bv_n_Weinstein}-\eqref{bv_n_properties_2}. From \eqref{properties_v} and the Brezis-Lieb lemma \cite{Brezis-Lieb1983PAMS} one has then
	\begin{equation}\label{chain_0}
    	\begin{split}
        	C_{N,p}&=\lim_{n\to+\infty}\frac{\|v_n\|_p^p}{\|v_n\|_{D^{2,2}}^{\frac{N(p-2)-4}2}\|v_n\|_{D^{1,2}}^{\frac{8-N(p-2)}2}\|v_n\|_2^{p-2}}\\
        	&=\lim_{n\to+\infty}\frac{\|v_n\|_p^p}{\|v_n\|_2^{p-2}}\\
        	&=\lim_{n\to+\infty}\frac{\|v\|_p^p}{\|v_n\|_2^{p-2}}+\lim_{n\to+\infty}\frac{\|v_n-v\|_p^p}{\|v_n\|_2^{p-2}}\,.
    	\end{split}
	\end{equation}
	Since
	\begin{equation}\label{L2_ineq}
	   \frac{\|v\|_2}{\|v_n\|_2}\leqslant1\quad\mbox{and}\quad\frac{\|v_n-v\|_2}{\|v_n\|_2}\leqslant1\,,
	\end{equation}
	applying \eqref{1.4} we deduce
	\begin{equation}\label{chain_1}
    	\begin{split}
        	C_{N,p}&\leqslant C_{N,p}\|v\|_{D^{2,2}}^{\frac{N(p-2)-4}2}\|v\|_{D^{1,2}}^{\frac{8-N(p-2)}2}\|v\|_2^{p-2}\lim_{n\to+\infty}\|v_n\|_2^{-(p-2)}\\	&\quad+C_{N,p}\lim_{n\to+\infty}\|v_n-v\|_{D^{2,2}}^{\frac{N(p-2)-4}2}\|v_n-v\|_{D^{1,2}}^{\frac{8-N(p-2)}2}\|v_n-v\|_2^{p-2}\|v_n\|_2^{-(p-2)}\\
        	&\leqslant C_{N,p}\|v\|_{D^{2,2}}^{\frac{N(p-2)-4}2}\|v\|_{D^{1,2}}^{\frac{8-N(p-2)}2}+C_{N,p}\lim_{n\to+\infty}\|v_n-v\|_{D^{2,2}}^{\frac{N(p-2)-4}2}\|v_n-v\|_{D^{1,2}}^{\frac{8-N(p-2)}2}.
    	\end{split}
	\end{equation}
	Noticing that
	$$\frac{N(p-2)-4}2+\frac{8-N(p-2)}2=2\,,$$
	we may use the Young inequality and get
	\begin{equation}\label{chain_2}
	\begin{split}
	C_{N,p}&\leqslant C_{N,p}\left(\frac{N(p-2)-4}4\|v\|_{D^{2,2}}^2+\frac{8-N(p-2)}4\|v\|_{D^{1,2}}^2\right)\\
	&\quad+C_{N,p}\lim_{n\to+\infty}\left(\frac{N(p-2)-4}4\|v_n-v\|_{D^{2,2}}^2+\frac{8-N(p-2)}4\|v_n-v\|_{D^{1,2}}^2\right)\\
	&=C_{N,p}\frac{N(p-2)-4}4\lim_{n\to+\infty}\left(\|v\|_{D^{2,2}}^2+\|v_n-v\|_{D^{2,2}}^2\right)\\
	&\quad+C_{N,p}\frac{8-N(p-2)}4\lim_{n\to+\infty}\left(\|v\|_{D^{1,2}}^2+\|v_n-v\|_{D^{1,2}}^2\right)\\
	&=C_{N,p}\lim_{n\to+\infty}\left(\|v_n\|_{D^{2,2}}^2+\|v_n\|_{D^{1,2}}^2\right)=C_{N,p}\,,
	\end{split}
	\end{equation}
	by \eqref{bv_n_properties_1} and again the Brezis-Lieb lemma. This implies that all the inequalities in \eqref{chain_0}-\eqref{chain_2} have to be indeed equalities. In particular, from \eqref{L2_ineq} we infer that
	\begin{equation*}
	\|v\|_2=0\qquad\mathrm{or}\qquad\lim_{n\to+\infty}\|v_n-v\|_2=0\,.
	\end{equation*}
	Since we know $v\not\equiv0$ by \eqref{v}, then necessarily
	\begin{equation}\label{L^2_limit}
	\lim_{n\to+\infty}\|v_n-v\|_2=0\,.
	\end{equation}
	Hence, using the interpolation inequality
	\begin{equation*}
	\begin{split}
	\|v_n-v\|_{D^{1,2}}&\leq C\|v_n-v\|_2^{1/2}\|v_n-v\|_{D^{2,2}}^{1/2}\\
	&=o_n(1)\left(\|v_n\|_{D^{2,2}}+\|v\|_{D^{2,2}}\right)=o_n(1)
	\end{split}
	\end{equation*}
	by \eqref{properties_v} and since $v\in H^2(\R^N)$. Hence, from the last line of \eqref{chain_1} we infer
	\begin{equation*}
	\|v\|_{D^{2,2}}^{\frac{N(p-2)-4}2}\|v\|_{D^{1,2}}^{\frac{8-N(p-2)}2}=1\,.
	\end{equation*}
	Combined with the fact that $\|v\|_{D^{2,2}}\leq\displaystyle{\liminf_{n\to+\infty}\,}\|v_n\|_{D^{2,2}}=1$ and $\|v\|_{D^{1,2}}\leq\displaystyle{\liminf_{n\to+\infty}\,}\|v_n\|_{D^{1,2}}=1$ by \eqref{bv_n_properties_1}, this yields $\|v\|_{D^{2,2}}=1$ and $\|v\|_{D^{1,2}}=1$. Therefore, again by the Brezis-Lieb lemma, one infers that
	$$\lim_{n\to+\infty}\|v_n-v\|_{D^{2,2}}=0\quad\mbox{and}\quad\lim_{n\to+\infty}\|v_n-v\|_{D^{1,2}}^2=0$$
	which, together with \eqref{L^2_limit}, yield $v_n\to v$ in $H^2(\R^N)$, and therefore, by continuous embedding, also in $L^p(\R^N)$. Since $\{v_n\}$ was a minimizing sequence for the Weinstein functional, such convergences imply that indeed $v$ is a nontrivial minimizer for $W_p$.
\end{proof}

\begin{lemma}\label{Lemma_eq}
	Let $v\in H^2(\R^N)$ be an optimizer for inequality \eqref{1.4} such that \eqref{properties_v} holds. Then there exist parameters $\lambda,\mu\neq0$ depending only on $N$ and $p$ such that
	\begin{equation}\label{Q}
	   Q:=\lambda v(\mu\,\cdot)
	\end{equation}
	weakly solves
	\begin{equation}\label{eq_Q}
	   \varepsilon\Delta^2Q-\Delta Q+\omega Q=|Q|^{p-2}Q\quad\mbox{in}\ \,\R^N,
	\end{equation}
	for $\omega>0$ defined by \eqref{omega}.
\end{lemma}

\begin{proof}
	If $v$ is an optimizer of \eqref{1.4} then it is a minimizer of $W_p$. Hence, $v$ satisfies
	\begin{equation*}
	   \left.\frac\dd{\dd t}\right|_{t=0}W_p(v+t\phi)=0
	\end{equation*}
	for all $\phi\in C^{\infty}_0(\R^N)$, from which, setting
	\begin{equation}\label{alpha_beta}
	   \alpha:=\frac{N(p-2)-4}{2}\quad\ \mbox{and}\ \quad\beta:=\frac{8-N(p-2)}{2}\,,
	\end{equation}
	it is standard to prove that $v$ is a weak solution of
	\begin{equation*}
	   \frac\alpha{\|\Delta v\|_2^2}\Delta^2v-\frac\beta{\|\nabla v\|_2^2}\Delta v+\frac{p-2}{\|v\|_2^2}v=\frac p{\|v\|_p^p}|v|^{p-2}v\quad\text{in }\R^N.
	\end{equation*}
	Hence by \eqref{properties_v} and since for $k\in\N$ one has $\Delta^kv(x)=\lambda^{-1}\mu^{-2k}\Delta^kQ(\frac x\mu)$, the corresponding equation that $Q$ satisfies is
	\begin{equation}\label{eq_Q_first}
	   \frac\alpha\beta\mu^{-2}\Delta^2Q-\Delta Q+\frac{(p-2)\mu^2}{\beta\|v\|_2^2}\,Q=\frac{p\lambda^{2-p}\mu^2}{\beta\|v\|_p^p}|Q|^{p-2}Q\quad\mbox{in }\R^N.
	\end{equation}
	Choosing now
	\begin{equation*}
	   \mu=\sqrt{\frac\alpha{\beta\varepsilon}}\quad\ \mbox{and}\ \quad\lambda=\left(\frac{p\alpha}{\beta^2\varepsilon\|v\|_p^p}\right)^\frac1{p-2},
	\end{equation*}
	then equation \eqref{eq_Q_first} reduces to \eqref{eq_Q} where $\omega$ is given by
	$$\omega=\frac{(p-2)\alpha}{\beta^2\varepsilon\|v\|_2^2}\,,$$
    that is \eqref{omega} by \eqref{alpha_beta}.
\end{proof}

\begin{proof}[Proof of Proposition \ref{Theorem4.1}]
	It directly follows from Lemmas \ref{Lemma_scaling}-\ref{Lemma_eq}. Note that $Q$ given by Lemma \ref{Lemma_eq} is a minimizer of the Weinstein functional, since $W_p$ is invariant under the scaling used to define $Q$ from $u$ in \eqref{Q}, as we proved in \eqref{Wp_invariant}.
\end{proof}

\texorpdfstring{$c_\varepsilon$}{c_epsilon}

\subsection{Any optimizer is an energy GSS with critical mass \texorpdfstring{$\boldsymbol c_\varepsilon$}{c_epsilon} (Proof of Proposition \ref{Theorem4.2})}

\begin{lemma}
	For $Q$ obtained in Lemma \ref{Lemma_eq}, one has the following identities:
	\begin{equation}\label{Q_norms_nabla_delta}
	   \int_{\R^N}|\nabla Q|^2\dd x=\frac{8-N(p-2)}{N(p-2)-4}\,\varepsilon\int_{\R^N}|\Delta Q|^2\dd x\,,
	\end{equation}
	\begin{equation}\label{Q_norms_nabla_p}
	   \int_{\R^N}|\nabla Q|^2\dd x=\frac{8-N(p-2)}{2p}\int_{\R^N}|Q|^p\dd x\,,
	\end{equation}
	\begin{equation}\label{Q_norms_Delta_p}
	   \int_{\R^N}|\Delta Q|^2\dd x=\frac{N(p-2)-4}{2p}\,\varepsilon^{-1}\int_{\R^N}|Q|^p\dd x\,.
	\end{equation}
	Moreover,
	\begin{equation}\label{Q_identity}
	   \frac{\varepsilon}2\int_{\R^N}|\Delta Q|^2\dd x+\frac12\int_{\R^N}|\nabla Q|^2\dd x-\frac1p\int_{\R^N}|Q|^p\dd x=0
	\end{equation}
	and
	\begin{equation}\label{Q_mass}
	   \int_{\R^N}|Q|^2\dd x\geqslant c_\varepsilon\,.
	\end{equation}
\end{lemma}

\begin{proof}
	Recalling that $Q$ was defined by \eqref{Q}, and that \eqref{properties_v} holds for $v$, we compute
	\begin{equation*}
    	\begin{split}
    	   \frac{\int_{\R^N}|\nabla Q|^2\dd x}{\int_{\R^N}|\Delta Q|^2\dd x}&=\mu^{-2}\frac{\int_{\R^N}|\nabla v|^2\dd x}{\int_{\R^N}|\Delta v|^2\dd x}=\mu^{-2}=\frac{\beta\varepsilon}\alpha=\frac{8-N(p-2)}{N(p-2)-4}\varepsilon
    	\end{split}
	\end{equation*}
	by \eqref{alpha_beta}. Similarly,
	\begin{equation*}
    	\begin{split}
    	   \frac{\int_{\R^N}|\nabla Q|^2\dd x}{\int_{\R^N}|Q|^p\dd x}&=\frac{\lambda^2\mu^{2-N}\int_{\R^N}|\nabla v|^2\dd x}{\lambda^p\mu^{-N}\int_{\R^N}|v|^p\dd x}=\frac\beta p=\frac{8-N(p-2)}{2p}\,.
    	\end{split}
	\end{equation*}
	The third identity plainly follows from the first two. Hence, using \eqref{Q_norms_nabla_p}-\eqref{Q_norms_Delta_p}, we get
	\begin{equation*}
    	\begin{split}
    	   &\frac{\varepsilon}2\int_{\R^N}|\Delta Q|^2\dd x+\frac12\int_{\R^N}|\nabla Q|^2\dd x-\frac1p\int_{\R^N}|Q|^p\dd x\\
    	   &=\frac{N(p-2)-4}{4p}\int_{\R^N}|Q|^p\dd x+\frac{8-N(p-2)}{4p}\int_{\R^N}|Q|^p\dd x-\frac1p\int_{\R^N}|Q|^p\dd x=0\,.
    	\end{split}
	\end{equation*}
	Since by Theorem \ref{Theorem4.1} $Q$ is a GSS of \eqref{eq_Q'}, necessary its norm has to be bigger to or equal than $c_\varepsilon$ by Theorem \ref{TheoremB}, that is \eqref{Q_mass}.
\end{proof}

\begin{proof}[Proof of Proposition \ref{Theorem4.2}]
	The fact that $E_\varepsilon(Q)=0$ follows from \eqref{Q_identity}. Now we show that $Q$ has critical mass. Since $Q$ is a minimizer of the Weinstein functional $W_p$, applying \eqref{Q_norms_nabla_delta}-\eqref{Q_norms_Delta_p}, one has
	\begin{equation*}
    	\begin{split}
    	   C_{N,p}&=\frac{\|Q\|_p^p}{\|Q\|_{D^{2,2}}^{\frac{N(p-2)-4}2}\|Q\|_{D^{1,2}}^{\frac{8-N(p-2)}2}\|Q\|_2^{p-2}}\\
    	   &=\frac{\|Q\|_p^p}{\|Q\|_2^{p-2}}\left[\frac{N(p-2)-4}{2p}\varepsilon^{-1}\int_{\R^N}|Q|^p\dd x\right]^{-\frac{N(p-2)-4}4}\\
    	   &\quad\times\left[\frac{8-N(p-2)}{2p}\int_{\R^N}|Q|^p\dd x\right]^{-\frac{8-N(p-2)}4}\\
    	   &=\left[\frac{2p}{N(p-2)-4}\right]^{\frac{N(p-2)-4}4}\left[\frac{2p}{8-N(p-2)}\right]^{\frac{8-N(p-2)}4}\frac{\varepsilon^{\frac{N(p-2)-4}4}}{\|Q\|_2^{p-2}}\,.
    	\end{split}
	\end{equation*}
	Combining this with \eqref{c_epsilon}, one easily gets $\|Q\|_2^2=c_\varepsilon$.
\end{proof}

\subsection{Any energy GSS with critical mass \texorpdfstring{$\boldsymbol c_\varepsilon$}{c_epsilon} is an optimizer for \texorpdfstring{\eqref{1.4}}{GN_C} (Proof of Proposition \ref{Theorem4.4})}

\begin{proof}[Proof of Proposition \ref{Theorem4.4}]
	Let $(u_{c_\varepsilon},\omega_{c_\varepsilon})\in H^2(\R^N)\times\R$ be a solution of \eqref{BS} with $c=c_\varepsilon$.
	From Theorem \ref{TheoremB}, we know that
	\begin{equation*}
	   0=E(u_{c_\varepsilon})=m_\varepsilon(c_\varepsilon)=\inf_{u\in S_{c_\varepsilon}}E(u)\,.
	\end{equation*}
	Set $u_{c_\varepsilon,t}(x):=t^{\frac N2}u_{c_\varepsilon}(tx)$. Then $u_{c_\varepsilon,t}\in S_{c_\varepsilon}$ for all $t>0$ and
	\begin{equation*}
		E(u_{c_\varepsilon,t})=\varepsilon\frac{t^4}2\|u_{c_\varepsilon}\|_{D^{2,2}}^2+\frac{t^2}2\|u_{c_\varepsilon}\|_{D^{1,2}}^2-\frac{t^{\frac{N(p-2)}2}}{p}\int_{\R^N}|u_{c_\varepsilon}|^p\dd x\,.
	\end{equation*}
	Setting
	\begin{equation*}
		h_\varepsilon(t):=\frac{E(u_{c_\varepsilon,t})}{t^2}=\varepsilon\frac{t^2}2\|u_{c_\varepsilon}\|_{D^{2,2}}^2+\frac12\|u_{c_\varepsilon}\|_{D^{1,2}}^2-\frac{t^{\frac{N(p-2)}2-2}}p\int_{\R^N}|u_{c_\varepsilon}|^p\dd x\,,
	\end{equation*}
	it is easy to show that
	\begin{equation*}
		t=t_\varepsilon:=\left[\frac{N(p-2)-4}{2p}\cdot\frac{\int_{\R^N}|u_{c_\varepsilon}|^p\dd x}{\|u_{c_\varepsilon}\|_{D^{2,2}}^2}\varepsilon^{-1}\right]^{\frac2{8-N(p-2)}}
	\end{equation*}
    is the unique critical point of $h_\varepsilon$, which is a minimum. From $E(u_{c_\varepsilon})=0$, we infer then $t_\varepsilon=1$, that is
	\begin{equation}\label{7.1}
	   \|u_{c_\varepsilon}\|_{D^{2,2}}^2=\frac{N(p-2)-4}{2p}\varepsilon^{-1}\|u_{c_\varepsilon}\|_p^p\,,
	\end{equation}
    which in turn implies
	\begin{equation}\label{7.2}
    	\begin{split}
        	\|u_{c_\varepsilon}\|_{D^{1,2}}^2=&\frac2p\|u_{c_\varepsilon}\|_p^p-\varepsilon\|u_{c_\varepsilon}\|_{D^{2,2}}^2\\
        	&=\left(\frac2{p}-\frac{N(p-2)-4}{2p}\right)\|u_{c_\varepsilon}\|_p^p\\
        	&=\frac{8-N(p-2)}{2p}\|u_{c_\varepsilon}\|_p^p\,,
    	\end{split}
	\end{equation}
	having exploited again $E_\varepsilon(u_{c_\varepsilon})=0$. Combining \eqref{7.1} and \eqref{7.2}, one has
	\begin{equation*}
    	\begin{split}
        	\|u_{c_\varepsilon}\|_{D^{2,2}}^{\frac{N(p-2)-4}2}&\|u_{c_\varepsilon}\|_{D^{1,2}}^{\frac{8-N(p-2)}2}=\left[\frac{N(p-2)-4}{2p}\varepsilon^{-1}\|u_{c_\varepsilon}\|_p^p\right]^{\frac{N(p-2)-4}4}\!\left[\frac{8-N(p-2)}{2p}\|u_{c_\varepsilon}\|_p^p\right]^{\frac{8-N(p-2)}4}\\
        	&=\left[\frac{N(p-2)-4}{2p}\right]^{\frac{N(p-2)-4}4}\!\left[\frac{8-N(p-2)}{2p}\right]^{\frac{8-N(p-2)}4}\!\!\varepsilon^{\frac{4-N(p-2)}4}\|u_{c_\varepsilon}\|_p^p\\
        	&=C_{N,p}^{-1}c_\varepsilon^{-\frac{p-2}2}\|u_{c_\varepsilon}\|_p^p\\
        	&=C_{N,p}^{-1}\|u_{c_\varepsilon}\|_2^{-(p-2)}\|u_{c_\varepsilon}\|_p^p\,,
    	\end{split}
	\end{equation*}
	by Proposition \ref{Theorem4.2} and $\|u_{c_\varepsilon}\|_2^2=c_\varepsilon$. This shows that $u_{c_\varepsilon}$ is an optimizer for the Weinstein functional and thus for \eqref{1.4}.
\end{proof}

\section{On the relationship between action GSS and energy GSS:\\ Proof of Theorem \ref{Theorem1.3}}

For $u\in H^2(\R^N)$ and $t>0$, set
$$u^t:=t^Nu(t\,\cdot)\,.$$
First we show that solutions of \eqref{BS} maximize the action $I_\omega$ within the set $\{u^t\}_{t>0}\,$. 
\begin{lemma}\label{Lemma5.1}
	Let $p\in(2+\frac4N,2+\frac8N)$, $\varepsilon>0$, and $u\in H^2(\R^N)$ be any weak solution of \eqref{BS}. Then
	\begin{equation*}
		I_\omega(u)\geqslant I_\omega(u^t)\qquad\mbox{for all}\,\ t>0\,.
	\end{equation*}
\end{lemma}

\begin{proof}
	Since $u\in H^2(\R^N)$ solves of \eqref{BS}, it satisfies
	\begin{equation*}
		\langle I_\omega'(u),u\rangle=\varepsilon\|u\|^2_{D^{2,2}}+\|u\|^2_{D^{1,2}}+\omega\|u\|_2^2-\|u\|_p^p=0
	\end{equation*}
	and moreover it is well-known that $u$ belongs to the zero set of the Poho\v zaev functional associated to $I_\omega$, namely
	\begin{equation*}
		P_\omega(u):=\varepsilon\frac{N-4}2\|u\|^2_{D^{2,2}}+\frac{N-2}2\|u\|^2_{D^{1,2}}+\omega\frac{N}2\|u\|_2^2-\frac Np\|u\|_p^p=0\,,
	\end{equation*}
	see \cite[Lemma 2.1]{Bonheure-Casteras-Gou-Jeanjean2019IMRN}. This implies that
	\begin{equation*}
		\begin{split}
			0&=N\langle I_\omega'(u),u\rangle-P_\omega(u)\\
			&=\varepsilon\frac{N+4}2\|u\|^2_{D^{2,2}}+\frac{N+2}2\|u\|^2_{D^{1,2}}+\omega\frac{N}2\|u\|_2^2-\frac{N(p-1)}{p}\|u\|_p^p\,.
		\end{split}
	\end{equation*}
	Hence we can compute
	\begin{equation}\label{I-I}
		\begin{split}
			I_\omega(u)-I_\omega(u^t)&=\frac{1-t^{N+4}}2\varepsilon\|u\|^2_{D^{2,2}}+\frac{1-t^{N+2}}2\|u\|^2_{D^{1,2}}\\
            &\quad+\frac{1-t^{N}}2\omega\|u\|_2^2-\frac{1-t^{N(p-1)}}{p}\|u\|_p^p\\
			&=\frac{1-t^{N+2}}2\|u\|^2_{D^{1,2}}-(1-t^{N+4})\frac{N+2}{2(N+4)}\|u\|^2_{D^{1,2}}\\
            &\quad+\frac{1-t^{N}}2\omega\|u\|_2^2-\omega(1-t^{N+4})\frac{N}{2(N+4)}\|u\|_2^2\\
			&\quad+(1-t^{N+4})\frac{N(p-1)}{p(N+4)}\|u\|_p^p-\frac{1-t^{N(p-1)}}{p}\|u\|_p^p\\
			&=\frac{2-(N+4)t^{N+2}+(N+2)t^{N+4}}{2(N+4)}\|u\|^2_{D^{1,2}}\\
			&\quad+\frac{4-(N+4)t^{N}+Nt^{N+4}}{2(N+4)}\omega\|u\|_2^2\\
			&\quad+\frac{N(p-2)-4-N(p-1)t^{N+4}+(N+4)t^{N(p-1)}}{p(N+4)}\|u\|_p^p\,.
		\end{split}
	\end{equation}
	Set now
	\begin{equation*}
		\begin{split}
			g_1(t)&:=2-(N+4)t^{N+2}+(N+2)t^{N+4}\\
			g_2(t)&:=4-(N+4)t^{N}+Nt^{N+4}\\
			g_3(t)&:=N(p-2)-4-N(p-1)t^{N+4}+(N+4)t^{N(p-1)}\,,
		\end{split}
	\end{equation*}
	from which
	\begin{equation*}
		\begin{split}
			g_1'(t)&=(N+4)(N+2)t^{N+1}[t^2-1]\\
			g_2'(t)&=N(N+4)t^{N-1}[t^{4}-1]\\
			g_3'(t)&=-N(p-1)(N+4)t^{N+3}+N(p-1)(N+4)t^{N(p-1)-1}\\
			&=N(p-1)(N+4)t^{N+3}[t^{N(p-2)-4}-1]\,.
		\end{split}
	\end{equation*}
	Since $p>2+\frac4N$ then $g_1(1)=g_2(1)=g_3(1)=0$ and $1$ is the global minimum point of $g_i$, $i\in\{1,2,3\}$ on $(0,+\infty)$. This implies that the right-hand side of \eqref{I-I} is nonnegative, with equality if and only if $t=1$, that is $u^t\equiv u$.
\end{proof}

\begin{lemma}[Energy GSS $\Rightarrow$ Action GSS]\label{Lemma5.3}
	For any $\varepsilon>0$, let $u_{c_\varepsilon}\in H^2(\R^N)$ be any energy GSS of \eqref{BS} with $c=c_\varepsilon$. Then $u_{c_\varepsilon}$ is an action GSS of \eqref{BS} for $\omega=\omega(\varepsilon)$.
\end{lemma}

\begin{proof}
    By Theorem \ref{Theorem1.2} we know that $u_{c_\varepsilon}$ is a weak solution of \eqref{BS} for $\omega=\omega(\varepsilon)$ defined in \eqref{omega}. Then $u_{c_\varepsilon}\in\mathcal N$, see \eqref{Nehari}, and so
    \begin{equation}\label{5.1}
    	\sigma_\omega\leqslant I_\omega(u_{c_\varepsilon})\,,
    \end{equation}
    where $\sigma_\omega$ is defined in \eqref{sigma_omega}. Our goal is to prove $I_\omega(u_{\varepsilon})=\sigma_\omega$. Now, let $w_{\varepsilon}$ be any action GSS of \eqref{BS} for $\omega=\omega(\varepsilon)$. Then by Lemma \ref{Lemma5.1}, we have
    \begin{equation*}
    	I_\omega(w_\varepsilon)\geqslant I_\omega(w_\varepsilon^t)
    \end{equation*}
    for all $t>0$, where $w_\varepsilon^t:=t^Nw_\varepsilon(t\,\cdot)$. Choosing $t=t_\varepsilon:=\left(\frac{c_\varepsilon}{\|w_{\varepsilon}\|_2^2}\right)^\frac1N$, one has
    \begin{equation}\label{5.2bis}
    	\|w_{\varepsilon}^{t_{\varepsilon}}\|_2^2=t_{\varepsilon}^{N}\|w_{\varepsilon}\|_2^2=c_\varepsilon\qquad\mbox{and}\qquad E_{\varepsilon}(u_{c_\varepsilon})\leqslant E_{\varepsilon}(w_{\varepsilon}^{t_{\varepsilon}})\,.
    \end{equation}
    It follows from \eqref{5.1}-\eqref{5.2bis} that
    \begin{equation*}
    	\begin{split}
    		\sigma_\omega\leqslant I_\omega(u_{c_\varepsilon})&=E_{\varepsilon}(u_{c_\varepsilon})+\frac{\omega(\varepsilon)}2c_\varepsilon\\
    		&\leqslant E_{\varepsilon}(w_{\varepsilon}^{t_{\varepsilon}})+\frac{\omega(\varepsilon)}2\|w_{\varepsilon}^{t_{\varepsilon}}\|_2^2\\
    		&=I_\omega(w_{\varepsilon}^{t_{\varepsilon}})\leqslant I_\omega(w_{\varepsilon})=\sigma_\omega\,.
    	\end{split}
    \end{equation*}
    This clearly shows $I_\omega(u_{c_\varepsilon})=\sigma_\omega\,$.
\end{proof}

\begin{lemma}[Action GSS $\Rightarrow$ Energy GSS]\label{Lemma5.4}
    For any $\varepsilon>0$, let $w_{\varepsilon}$ be any action GSS of \eqref{BS} for $\omega=\omega(\varepsilon)$. Then $w_{\varepsilon}$ is an energy GSS of \eqref{BS} for the mass $c_\varepsilon$ with Lagrange multiplier $\omega=\omega(\varepsilon)$.
\end{lemma}

\begin{proof}
	Since $w_{\varepsilon}$ is an action GSS of \eqref{BS} for $\omega=\omega(\varepsilon)$, in particular is a weak solution of \eqref{BS}, hence Lemma \ref{Lemma5.1} implies that
	\begin{equation*}
		I_\omega(w_{\varepsilon})\geqslant I_\omega(w_\varepsilon^t)\quad\mbox{for all}\ \,t>0\,.
	\end{equation*}
	Let $u_{c_\varepsilon}$ be any energy GSS of \eqref{BS} for the mass $c_\varepsilon$ whose Lagrange multiplier is $\omega=\omega(\varepsilon)$. Then $u_{c_\varepsilon}\in\mathcal N$,
	\begin{equation*}
		\sigma_\omega\leqslant I_\omega(u_{c_\varepsilon})\,,
	\end{equation*}
	and
	\begin{equation*}
		0=E_{\varepsilon}(u_{c_\varepsilon})\leqslant E_{\varepsilon}(u),~~\mathrm{for~~any}~~u\in S_{c_\varepsilon}.
	\end{equation*}
	by Lemma \ref{Lemma_Bonheure}. Taking $t=\tilde t_\varepsilon:=\left(\frac{c_\varepsilon}{\|w_{\varepsilon}\|_2^2}\right)^\frac1N$,
	we find that
	\begin{equation*}
		\|w_{\varepsilon}^{\tilde t_{\varepsilon}}\|_2^2=\tilde t_{\varepsilon}^N\|w_{\varepsilon}\|_2^2=c_\varepsilon\qquad\mbox{and so}\qquad E_{\varepsilon}(u_{c_\varepsilon})\leqslant E_{\varepsilon}(w_{\varepsilon}^{\tilde{t}_{\varepsilon}})\,.
	\end{equation*}
	These show that
	\begin{equation}\label{5.3}
		\begin{split}
			\sigma_\omega\leqslant I_\omega(u_{c_\varepsilon})&=E_{\varepsilon}(u_{c_\varepsilon})+\frac{\omega(\varepsilon)}2c_\varepsilon\\
			&\leqslant E_{\varepsilon}(w_{\varepsilon}^{\tilde t_{\varepsilon}})+\frac{\omega(\varepsilon)}2\|w_{\varepsilon}^{\tilde{t}_{\varepsilon}}\|_2^2\\
			&=I_\omega(w_{\varepsilon}^{\tilde t_{\varepsilon}})\leqslant
			I_\omega(w_{\varepsilon})=\sigma_\omega\,.
		\end{split}
	\end{equation}
	In particular
	\begin{equation*}
		I_\omega(w_{\varepsilon}^{\tilde{t}_{\varepsilon}})=I_\omega(w_{\varepsilon})\,,
	\end{equation*}
	and thus
	\begin{equation*}
		\tilde t_{\varepsilon}=1\,,\qquad\mbox{that is}\qquad\|w_{\varepsilon}\|_2^2=c_\varepsilon\,.
	\end{equation*}
	Going back again to \eqref{5.3}, we deduce that
	\begin{equation*}
		E_{\varepsilon}(w_\varepsilon)=E_{\varepsilon}(u_{c_\varepsilon})=0\,.
	\end{equation*}
	By Lemma \ref{Lemma_Bonheure}, this implies that $w_\varepsilon$ is an energy GSS for the mass $c_\varepsilon$ whose Lagrange multiplier is $\omega=\omega(\varepsilon)$.
\end{proof}

\begin{proof}[Proof of Theorem \ref{Theorem1.3}]
	It follows by combining Lemmas \ref{Lemma5.3}, \ref{Lemma5.4}, and Theorem \ref{Theorem1.2}.
\end{proof}

\section*{Acknowledgements}
Z. Liu is supported by the NSFC (No.12571188) and Guangdong Basic and Applied Basic Research Foundation (Nos. 2024A1515012704). Y. Su is supported by the Natural Science Research Project of Anhui Educational Committee (Grant No.  2023AH040155). G. Romani is a member of the {\em Gruppo Nazionale per l'Analisi Ma\-te\-ma\-ti\-ca, la Probabilit\`a e le loro Applicazioni} (GNAMPA) of the {\em Istituto Nazionale di Alta Matematica} (INdAM), and partially supported by INdAM-GNAMPA Project 2026 titled \textit{Structural degeneracy and criticality in (sub)elliptic PDEs} (CUP E53C25002010001).

\section*{Conflict of interest}
\noindent The author declares that he has no known competing financial interests or personal relationships that could have appeared to influence the work reported in
this article.

\section*{Data availability statement}
\noindent No data were used for the research described in the article.


\small
\begin{thebibliography}{10}

\bibitem{Brezis-Lieb1983PAMS}
H. Br\'{e}zis, E. Lieb,
\href{https://doi.org/10.2307/2044999}{A relation between pointwise convergence of functions and convergence of functionals},
Proc. Am. Math. Soc. 88 (1983), 486-490.

\bibitem{Bonheure-Casteras-dosSantos-Nascimento2018SIAM}
D. Bonheure, J. Casteras,  E. dos Santos, R. Nascimento,
\href{https://doi.org/10.1137/17M1154138}{Orbitally stable standing waves of a mixed dispersion nonlinear Schr\"{o}dinger equation},
SIAM J. Math. Anal. 50 (2018), 5027-5071.

\bibitem{Bonheure-Casteras-Gou-Jeanjean2019IMRN}
D. Bonheure, J. Casteras, T. Gou, L. Jeanjean,
\href{https://doi.org/10.1093/imrn/rnx273}{Strong instability of ground states to a fourth order Schr\"{o}dinger equation},
Int. Math. Res. Not. 17 (2019), 5299-5315.

\bibitem{Bonheure-Casteras-Gou-Jeanjean2019TAMS}
D. Bonheure, J. Casteras,  T. Gou, L. Jeanjean,
\href{https://doi.org/10.1090/tran/7769}{Normalized solutions to the mixed dispersion nonlinear Schrdinger equation in the mass-critical and supercritical regime},
Trans. Am. Math. Soc. 372 (2019), 2167-2212.

\bibitem{Bonheure-Casteras-Mandel2019JLMS}
D. Bonheure, J. Casteras,  R. Mandel,
\href{https://doi.org/10.1112/jlms.12196}{On a fourth-order nonlinear Helmholtz equation},
J. London Math. Soc. 99 (2019), 831-852.

\bibitem{BN}
D. Bonheure, R. Nascimento,
\href{https://doi.org/10.1007/978-3-319-19902-3_4}{Waveguide solutions for a nonlinear Schr\"odinger equation with mixed dispersion}, A.N.
Carvalho et al. (eds.), Contributions to Nonlinear Elliptic Equations and Systems, Progress in Nonlinear Differential Equations and Their Applications 86, (2015), 31–53

\bibitem{Boulenger-Lenzmann2017AENS}
T. Boulenger, E. Lenzmann,
\href{https://doi.org/10.24033/asens.2326}{Blowup for biharmonic NLS},
Ann. Sci. \'{E}c. Norm. Sup\'{e}r. 50 (2017), 503-544.

\bibitem{Campos-Guzan2022CVPDE}
L. Campos, C. Guz\'{a}n,
\href{https://doi.org/10.1007/s00526-022-02256-x}{Scattering for the non-radial inhomogenous biharmonic NLS equation},
Calc. Var. Partial Differ. Equ. 61 (2022), 156.

\bibitem{CT}
D. Cassani, A. Tarsia,
\href{https://doi.org/10.1515/anona-2021-0210}{Maximum principle for higher order operators in general domains},
Adv. Nonlinear Anal. 11 (2022), no. 1, 655-671.

\bibitem{Cazenave2003BOOK}
T. Cazenave,
Semilinear Schr\"{o}dinger equations,
Courant Lecture Notes in Mathematics, 10
New York University,
Courant Institute of Mathematical Sciences, New York; American Mathematical Society, Providence, RI, 2003.

\bibitem{Cazenave-Lions1982CMP}
T. Cazenave, P. Lions,
\href{http://projecteuclid.org/euclid.cmp/1103921547}{Orbital stability of standing waves for some nonlinear Schr\"{o}dinger equations},
Commun. Math. Phys. 85 (1982), 549-561.

\bibitem{d'Avenia-Pomponio-Schino2023Nonlinearity}
P. d'Avenia, A. Pomponio, J. Schino,
\href{https://doi.org/10.1088/1361-6544/acb62d}{Radial and non-radial multiple solutions to a general mixed dispersion NLS equation},
Nonlinearity, 36 (2023), 1743-1775.

\bibitem{Dinh2021Nonlinearity}
V. Dinh,
\href{https://doi.org/10.1088/1361-6544/abcea5}{Dynamics of radial solutions for the focusing fourth-order nonlinear Schr\"{o}dinger equations},
Nonlinearity, 34 (2021), 776-821.

\bibitem{Dovetta-Serra-Tilli2020AM}
S. Dovetta, E. Serra,  P. Tilli,
\href{https://doi.org/10.1016/j.aim.2020.107352}{Uniqueness and nonuniqueness of prescribed mass NLS ground states on metric graphs},
Adv. Math. 374 (2020), 107352.

\bibitem{Dovetta-Serra-Tilli2023MA}
S. Dovetta, E. Serra,  P. Tilli,
\href{https://doi.org/10.1007/s00208-022-02382-z}{Action versus energy ground states in nonlinear Schr\"{o}dinger equations},
Math. Ann. 385 (2023), 1545-1576.

\bibitem{FY}
Z. Feng, Y. Su
\href{https://doi.org/10.1007/s00033-021-01643-2}{Ground state solution to the biharmonic equation},
Z. Angew. Math. Phys. 73 (2022), no. 1, Paper No. 15, 24 pp.

\bibitem{Fernandez-Jeanjean-Mandel-Maris2022JDE}
A. Fernandez, L. Jeanjean, R. Mandel, M. Maris,
\href{https://doi.org/10.1016/j.jde.2022.04.037}{Non-homogeneous Gagliardo-Nirenberg inequalities in and application to a biharmonic non-linear Schr\"{o}dinger equation},
J. Differ. Equ.  330 (2022), 1-65.

\bibitem{Fibich-Ilan-Papanicolaou2002SIAMJAM}
G. Fibich, B. Ilan, G. Papanicolaou,
\href{https://doi.org/10.1137/S0036139901387241}{Self-focusing with fourth-order dispersion},
SIAM J. Appl. Math. 62 (2002), 1437-1462.

\bibitem{Fibich-Ilan-Schochet2003Nonlinearity}
G. Fibich, B. Ilan, S. Schochet,
\href{https://doi.org/10.1088/0951-7715/16/5/314}{Critical exponents and collapse of nonlinear Schr\"{o}dinger equations with anisotropic fourth-order dispersion},
Nonlinearity, 16 (2003), 1809-1821.

\bibitem{GGS} F. Gazzola, H.-Ch. Grunau, G. Sweers,
\href{https://doi.org/10.1007/978-3-642-12245-3}{Polyharmonic boundary value problems}, Springer Lecture Notes in Mathematics n. 1991, 2010.

\bibitem{GRS}
H-C. Grunau, G. Romani, G. Sweers,
\href{https://doi.org/10.1007/s00208-020-02015-3}{Differences between fundamental solutions of general higher order elliptic operators and of products of second order operators}, Math. Ann. 381, 1031–1084 (2021).

\bibitem{Hajaiej-SongLJ2025MZ}
H. Hajaiej, L. Song,
\href{https://doi.org/10.1007/s00209-025-03705-x}{Uniqueness of normalized ground states for NLS models},
Math. Z. 310 (2025), 8.

\bibitem{JL}
L. Jeanjean, SS. Lu,
\href{https://doi.org/10.1007/s00526-022-02320-6}{On global minimizers for a mass constrained problem},
Calc. Var. Partial Differ. Equ. 61 (2022), 214.

\bibitem{Karpman1996PRE}
V. Karpman,
\href{https://doi.org/10.1016/0375-9601(95)00752-0}{Stabilization of soliton instabilities by higher order dispersion: KdV-type equations},
Physics Letters A, 210 (1996), 77-84.

\bibitem{Karpman-Shagalov2000PD}
V. Karpman,
A. Shagalov,
\href{https://doi.org/10.1016/S0167-2789(00)00078-6}{Stability of solitons described by nonlinear Schr\"{o}dinger-type equations with higher-order dispersion},
Physica D, 144 (2000), 194-210.

\bibitem{Lenzmann-Weth2023JAM}
E. Lenzmann, T. Weth,
\href{https://doi.org/10.1007/s11854-023-0311-2}{Symmetry breaking for ground states of biharmonic NLS via Fourier extension estimates},
J. Anal. Math. 152 (2024), 777-800.

\bibitem{Lenzmann-Sok2021IMRN}
E. Lenzmann, J. Sok,
\href{https://doi.org/10.1093/imrn/rnz274}{A sharp rearrangement principle in Fourier space and symmetry results for PDEs with arbitrary order},
Int. Math. Res. Not., (2021), 15040-15081.

\bibitem{LZZ}
T. Luo, S. Zheng, S. Zhu,
\href{https://doi.org/10.1007/s10473-023-0205-5}{The existence and stability of normalized solutions for a bi-harmonic nonlinear Schrödinger equation with mixed dispersion.}
Acta Math. Sci. Ser. B (Engl. Ed.) 43 (2023), no. 2, 539-563.

\bibitem{LTW}
X. Luo, Z. Tang, L. Wang,
\href{https://doi.org/10.1080/00036811.2023.2213243}{Infinitely many solutions for nonlinear fourth-order Schrödinger equations with mixed dispersion},
Appl. Anal. 103 (2024), no. 5, 898-926.

\bibitem{LuoX-YangT2023SCM}
X. Luo, T. Yang,
\href{https://doi.org/10.1007/s11425-022-1997-3}{Normalized solutions for a fourth-order Schr\"{o}dinger equation with a positive second-order dispersion coefficient},
Sci. China Math. 66 (2023), 1237-1262.

\bibitem{Lions} P.-L. Lions,
\href{https://www.numdam.org/item/AIHPC_1984__1_2_109_0/}{The concentration-compactness principle in the calculus of variations. The locally compact case. I.}
Ann. Inst. H. Poincaré Anal. Non Linéaire 1 (1984), no. 2, 109-145.

\bibitem{Ma}
C. Ma,
\href{https://doi.org/10.3934/era.2023191}{Normalized solutions for the mixed dispersion nonlinear Schrödinger equations with four types of potentials and mass subcritical growth},
Electron. Res. Arch. 31 (2023), no. 7, 3759-3775.

\bibitem{N}
L. Nirenberg,
\href{https://www.numdam.org/item/ASNSP_1959_3_13_2_115_0/}{On elliptic partial differential equations},
Ann. Scuola Norm. Sup. Pisa Cl. Sci. (3) 13 (1959), 115–162.

\bibitem{Pausader2007DPDE}
B. Pausader,
\href{https://doi.org/10.4310/DPDE.2007.v4.n3.a1}{Global well-posedness for energy critical fourth-order Schr\"{o}dinger equations in the radial case},
Dyn. Partial Differ. Equ. 4 (2007), 197-225.

\bibitem{Pausader2009JFA}
B. Pausader,
\href{https://doi.org/10.1016/j.jfa.2008.11.009}{The cubic fourth-order Schr\"{o}dinger equation},
J. Funct. Anal. 256 (2009), 2473-2517.

\bibitem{Pausader-ShaoSL2010JHDE}
B. Pausader, S. Shao,
\href{https://doi.org/10.1142/S0219891610002256}{The mass-critical fourth-order Schr\"{o}dinger equation in high dimensions},
J. Hyperbolic Differ. Equ. 7 (2010), 651-705.


\bibitem{Saanouni2021CVPDE}
T. Saanouni,
\href{https://doi.org/10.1007/s00526-021-01973-z}{Energy scattering for radial focusing inhomogeneous bi-harmonic Schr\"{o}dinger equations},
Calc. Var. Partial Differ. Equ. 60 (2021), 113.

\bibitem{Weinstein1983CMP}
M. Weinstein,
\href{http://projecteuclid.org/euclid.cmp/1103922134}{Nonlinear Schr\"{o}dinger equations and sharp interpolation estimates},
Commun. Math. Phys. 87 (1983), 567-576.

\end{thebibliography}
\end{document}